\documentclass{amsart}

\usepackage{amsfonts}
\usepackage{amsmath}
\usepackage{amsthm}
\usepackage{amssymb}
\usepackage{latexsym}

\newtheorem{teo}{Theorem}[section]
\newtheorem{lema}[teo]{Lemma}
\newtheorem{prop}[teo]{Proposition}

\theoremstyle{definition}

\numberwithin{equation}{section}

\begin{document}

\newcommand{\cc}{\mathfrak{c}}
\newcommand{\pp}{\mathfrak{p}}
\newcommand{\N}{\mathbb{N}}

\title{On the density of Banach $C(K)$ spaces with the Grothendieck property}
\author{Christina Brech}
\address{Departamento de  
Matem\'atica - Instituto de Matem\'atica e Estat\'\i stica - Universidade de 
S\~ao Paulo, Caixa Postal 66281 - CEP: 05315-970 - S\~ao Paulo, Brasil}
\email{kika@ime.usp.br}
\thanks{The author was supported by scholarship from FAPESP (02/04531-6). This
  paper is part of the author's M.A. thesis at the University of S\~ao
  Paulo, under the guidance of Prof. Piotr Koszmider. I wish to thank him
for the guidance and for the assistance during the preparation of this
paper.}
%
%
%
%
\begin{abstract} 
Using the method of forcing we prove that consistently there is a Banach space
of continuous functions on a compact Hausdorff space with the Grothendieck
property and with density less than the continuum. It follows that the
classical result stating that ``no nontrivial complemented subspace of a
Grothendieck $C(K)$ space is separable'' cannot be strengthened by replacing
``is separable'' by ``has density less than that of $l_\infty$'', without
using an additional set-theoretic assumption. Such a strengthening was proved
by Haydon, Levy and Odell, assuming Martin's axiom and the negation of the
continuum hypothesis. Moreover, our example shows that certain separation
properties of Boolean algebras are quite far from the Grothendieck property.
\end{abstract}

\maketitle

\section{Introduction}

For an infinite compact Hausdorff space $K$, let $C(K)$ be the Banach space of
the continuous real-valued functions on $K$, with the supremum norm. The
purpose of this work is to show that the existence of a Grothendieck $C(K)$
space with density less than the continuum (denoted by $\cc$) is independent
from the usual axioms of set theory. Recall that a Banach space $X$ is said to
be a Grothendieck space (see \cite{DiestelArtigo} for more details) whenever
each weak$^*$ convergent sequence in its dual $X^*$ converges weakly. To
obtain this independence result we make the following two assertions:
\[
\begin{array}{c}
\textit{If }\pp=\cc, \textit{ then every Grothendieck }
C(K)\textit{ space has density }\geq \cc.
\end{array}\eqno{(\mathrm{I})}
\]
\[
\begin{array}{c}
\textit{In a model obtained by forcing, there is a Grothendieck}\\
C(K)\textit{ space with density }< \cc.
\end{array}\eqno{(\mathrm{II})} 
\]
The main purpose of this work is to prove $(\mathrm{II})$, since
$(\mathrm{I})$ is already known: it follows from a result of
\cite{HaydonLevyOdell}. To present here a direct proof of $(\mathrm{I})$, we
define the cardinal $\pp$: we call $\pp$ the least infinite cardinal $\kappa$
for which there exists $(M_\alpha)_{\alpha<\kappa} \subseteq \wp(\N)$ such
that $\bigcap_{\alpha \in F}M_\alpha$ is infinite for all finite subsets $F$
of $\kappa$ and there is no infinite $M \subseteq \N$ such that $|M\setminus
M_\alpha| < \infty$ for all $\alpha <\kappa$. This means that we can in some
way diagonalize less than $\pp$ subsets of $\N$ which are finitely
compatible. It is known that $\omega_1 \leq \pp \leq \cc$ ($\omega_1$ is the
first uncountable cardinal) and that MA (Martin's axiom) implies that
$\pp=\cc$. We have that every infinite compact Hausdorff space with weight
less than $\pp$ has a nontrivial convergent sequence (see \cite{Fremlin},
Proposition 24 A) and therefore $C(K)$ is not a Grothendieck space (see the
proof of Theorem 9.5 of \cite{Mascioni}). So, assuming $\pp=\cc$ we have that
if $C(K)$ is a Grothendieck space, then $K$ has weight at least $\cc$ and by
the Stone-Weierstrass Theorem, $C(K)$ has density at least $\cc$.

It follows from a result of \cite{DiestelArtigo} that no nontrivial
complemented subspace of a Grothendieck $C(K)$ space is separable. A
strengthening of this statement follows from a result of
\cite{HaydonLevyOdell}, assuming MA (or simply $\pp=\cc$) and the negation of
CH (continuum hypothesis): each nontrivial complemented subspace of a
nonreflexive Grothendieck space (hence each nontrivial complemented subspace
of a Grothendieck $C(K)$ space) has density at least $\cc$. Our result shows
that we need an additional set-theoretic assumption to prove such
strengthening.

Pe\l czy\' nski asked (see \cite{Haydon}) if every Banach space of continuous
functions should contain either a complemented copy of $c_0$ or a
(complemented) copy of $l_\infty$. Talagrand (assuming CH, see
\cite{Talagrand}) and Haydon (without any additional hypothesis, see
\cite{Haydon}) answered this question negatively. Moreover, the space
constructed by Talagrand does not have a quotient space isomorphic to
$l_\infty$. On the other hand, Haydon, Levy and Odell proved in
\cite{HaydonLevyOdell} that $\pp=\cc$ and the negation of CH imply that every
Grothendieck $C(K)$ space has $l_\infty$ as a quotient. Our space has stronger
properties than that constructed by Talagrand: it is a Grothendieck $C(K)$
space or, equivalently, a Banach $C(K)$ space with no complemented copies of
$c_0$ (by a result of \cite{Schachermayer}) with density less than $\cc$,
which is the density of $l_\infty$ (and therefore it has no quotient
isomorphic to $l_\infty$).

Turning to properties of Boolean algebras, we would like to notice that there
are many of them which imply that $C(K)$ has the Grothendieck property, for
$K$ its Stone space. Some of them are the subsequential completeness property
(see \cite{Haydon}), subsequential interpolation property (see
\cite{Freniche}), etc. However, all of them imply also that the Boolean
algebra has cardinality at least $\cc$, which is not the case of ours. So, our
space is a Grothendieck $C(K)$ space, for $K$ the Stone space of a Boolean
algebra, which does not have such properties. This illustrates that these
properties are quite far from the Grothendieck property.

To show $(\mathrm{II})$ we will not make use of well-known axioms like CH or
$\pp=\cc$, as occurs in the results of Talagrand and Haydon, Levy and
Odell. Instead, we shall prove the consistency directly by forcing\footnote{A
  classical example of the use of forcing to obtain a result in analysis is
  the proof of the consistency of the automatic continuity of homomorphisms
  between Banach algebras (see \cite{DalesWoodin}).}. Using a product of Sacks
forcings (also known as the perfect set forcing) we obtain the model in which
there is a Grothendieck $C(K)$ space with density less than $\cc$. It would be
interesting to decide if axioms like $\pp<\cc$ or alike imply $(\mathrm{II})$
directly. Other applications of the Sacks forcing in analysis can be found in
\cite{CiesielskiPawlikowski} and \cite{Steprans}.

The idea of showing $(\mathrm{II})$ was motivated by a result of
\cite{JustPiotr}. In this work, Just and Koszmider showed that a certain
compact Hausdorff space $K$ with weight less than $\cc$ (and so, $C(K)$ has
density less than $\cc$) has no nontrivial convergent sequences. Although this
is not sufficient for $C(K)$ to be a Grothendieck space, it is
necessary. Thus, generalizing and modifying the methods used in
\cite{JustPiotr} we prove that $C(K)$ is a Grothendieck space. Moreover,
Schachermayer proved (see \cite{Schachermayer}) that a necessary (but not
sufficient) condition for a Boolean algebra to have the Grothendieck property
(that is, for $C(K)$ to be a Grothendieck space, where $K$ is its Stone space)
is that it is not a countable union of a strictly increasing sequence of
subalgebras. One of the results of \cite{JustPiotr} (which follows also
combining ours and that of \cite{Schachermayer}) is that our Boolean algebra
is not such a union.

In this paper, $\mathcal{B}$ denotes an infinite Boolean algebra and
$S(\mathcal{B})$ its Stone space. We use $\bigvee_{i\in I}b_i$ for the
supremum of the family $(b_i:i\in I) \subseteq \mathcal{B}$, if it exists and
we say that a family $A \subseteq \mathcal{B}$ is an antichain in
$\mathcal{B}$ if for each $a, b \in A$ with $a \neq b$ we have that $a \cdot b
= 0$. We denote by $\mu$ a real-valued finitely additive measure on
$\mathcal{B}$ and if $\mu$ is bounded, $|\mu|$ denotes the variation of
$\mu$. $K$ will always denote an infinite compact Hausdorff space and $Bor(K)$
the $\sigma$-algebra of its Borel sets. A Radon measure $\mu$ on $K$ is a
real-valued $\sigma$-additive bounded regular measure on $Bor(K)$ and $\Vert
\mu \Vert$ denotes its norm (see \cite{Schachermayer} and \cite{Semadeni} for
the definitions).

Let $\mathcal{B}$ be a Boolean algebra. In what follows we will identify the
Boolean algebra $\mathcal{B}$ with the Boolean algebra $Clop(S(\mathcal{B}))$
of the closed and open subsets of $S(\mathcal{B})$, using the Stone
duality. Recall that given a Radon measure $\mu$ on $K$, $\mu|_{\mathcal{B}}$
is a finitely additive measure on $\mathcal{B}$. On the other hand, if $\mu$
is a finitely additive measure on $\mathcal{B}$, then there is a unique Radon
measure $\tilde{\mu}$ on $K$ such that $\tilde{\mu}|_{\mathcal{B}}
=\mu$. Thus, we have a correspondence between finitely additive measures on
$\mathcal{B}$ and Radon measures on $K$ and we will identify them in the
sequel. Recall that the Riesz Representation Theorem guarantees that
$C(S(\mathcal{B}))^*$ (the dual space of $C(S(\mathcal{B}))$) is isometric to
the space of the Radon measures on $S(\mathcal{B})$. Hence we identify also
each Radon measure on $S(\mathcal{B})$ (and thus, each finitely additive
measure on $\mathcal{B}$) with the correspondent functional in
$C(S(\mathcal{B}))^*$.

We use a standard terminology for the Sacks forcing: we denote by $\mathbb{S}$
the Sacks forcing and given $s \in \mathbb{S}$ and $p \in s$, let $s|p = \{q
\in s: q\subseteq p \text{ or } p \subseteq q\} \in \mathbb{S}$. We denote by
$l(n,s)$ the $n$th forking level of $s\in \mathbb{S}$ and we say that $s
\leq_n t$ if $s \leq t$ and $l(n,s)=l(n,t)$ (see \cite{BaumgartnerSacks} for
the definitions).

Given a regular cardinal $\kappa$, we denote by $\mathbb{S}^\kappa$ the
product of $\kappa$ Sacks forcings and given $s \in \mathbb{S}^\kappa$, a
finite subset $F$ of $dom(s)$ and $n \in \N$, we denote by $l(F,n,s)$ the set
$\{\sigma: dom(\sigma) = F \text{ and for all }\alpha \in F, \sigma(\alpha)
\in l(n,s(\alpha))\}$. We say that $s \leq_{F,n} t$ if $s \leq t$ and
$l(F,n,s)=l(F,n,t)$. Finally, if $s \in \mathbb{S}^\kappa$, if $F$ is a finite
subset of $dom(s)$ and if $\sigma$ is a function with domain $F$ such that
$\sigma(\alpha) \in s(\alpha)$ for all $\alpha \in F$, let $s|\sigma \in
\mathbb{S}^\kappa$ be such that $(s|\sigma)(\alpha)=s(\alpha)$ for $\alpha \in
dom(s) \setminus F$ and $(s|\sigma)(\alpha)=s(\alpha)|\sigma(\alpha)$ for
$\alpha \in F$. We will need some results about this forcing, which are all
proved in \cite{BaumgartnerSacks}.

In Section 2 we present some combinatorial results needed for the proof of
$(\mathrm{II})$ and in Section 3 we present the proof of the main result. The
notation and terminology used are those of \cite{Diestel} for Grothendieck
spaces and those of \cite{BaumgartnerSacks} for Sacks forcing.

\section{Some combinatorial results}

In this section we present some combinatorial results, which will be necessary
in the proof of the main theorem. The following lemma is implicit in
\cite{Rosenthal1}.

\begin{lema}\label{combRosenthal}
Let $\mathcal{B}$ be a Boolean algebra and let $\kappa$ be an uncountable
cardinal. Let $(a_n :n \in \N)$ be an antichain in $\mathcal{B}$ and let
$(\mu_k: k \in \N)$ be a sequence in $C(S(\mathcal{B}))^*$. If $(N_\xi:\xi <
\kappa)$ is an almost disjoint family in $\wp(\N)$ (that is, a family of
infinite subsets of $\N$ such that for all $\xi<\xi'<\kappa$, $N_\xi \cap
N_{\xi'}$ is finite), then for all but countably many $\xi$'s we have that for
all $k \in \N$ and all $M \subseteq N_\xi$, if $\bigvee_{n \in M}a_n$ exists,
then
$$\mu_k(\bigvee_{n \in M}a_n)=\sum_{n \in M}\mu_k(a_n).$$
\end{lema}
\begin{proof}
Suppose that the lemma does not hold. Then there is an uncountable $X
\subseteq \kappa$ such that for each $\xi \in X$ there is $k_\xi \in \N$ and
an infinite set $M_\xi \subseteq N_\xi$ such that $\bigvee_{i \in M_\xi}a_i$
exists and $\sum_{i \in M_\xi} \mu_{k_\xi}(a_i) \neq \mu_{k_\xi}(\bigvee_{i
  \in M_\xi}a_i)$. We can assume without loss of generality that there are
natural numbers $k$ and $m$ such that for all $\xi \in X$,
$$|\mu_{k}(\bigvee_{i \in M_\xi}a_i)-\sum_{i \in M_\xi}\mu_{k}(a_i)|>
\frac{1}{m}.$$

Fix $\xi \in X$. Let $\delta_\xi = \mu_{k}(\bigvee_{i \in M_\xi}a_i)-\sum_{i
  \in M_\xi} \mu_{k}(a_i)$. Since $\mu_k(\bigvee_{i \in M_\xi, i >l}a_i)$
converges to $\delta_\xi$ as $l \rightarrow \infty$, there are arbitrarily
large $l \in \N$ such that
\[|\mu_k(\bigvee_{i\in M_\xi,i>l}a_i)-\delta_\xi|<|\delta_\xi|-
  \frac{1}{m}. \eqno{(*)}\] 

Let $n$ be a natural number greater than $m \cdot \Vert \mu_k\Vert$ and let
$\xi_1,\dots,\xi_n$ be different ordinals in $X$ such that $\delta_{\xi_j}$
are either all positive or all negative. For each $1 \leq j \leq n$, let $l_j$
be a natural number as in $(*)$ and such that $(M_{\xi_j} \setminus \{0,
\dots, l_j\})_{1 \leq j \leq n}$ are pairwise disjoint. Note that $(*)$
implies that $\mu_k(\bigvee_{i \in M_{\xi_j}, i>l_j}a_i )$ are also either all
positive or all negative and that $|\mu_k(\bigvee_{i \in M_{\xi_j}, i>l_j}a_i
)| > \frac{1}{m}$. Then,
$$|\mu_k(\bigvee_{1 \leq j \leq n} (\bigvee_{i \in M_{\xi_j}, i > l_j}a_i))| =
\sum_{j=1}^n |\mu_k(\bigvee_{i \in M_{\xi_j}, i>l_j}a_i )| \geq n \cdot
\frac{1}{m} > \Vert \mu_k\Vert,$$
a contradiction. Therefore the lemma is true.
\end{proof}

\begin{lema}\label{combinatorio}
Let $m$, $A$ and $N$ be natural numbers. For each $i<N$, let $G_i$ be a finite
subset of $\N$ with cardinality at least $A+m$. Fix an infinite $X \subseteq
\N$ and suppose that for each $i<N$ and each $k \in G_i$, $X_{k,i}\subseteq X$
is such that $X \subseteq \bigcup\{X_{k,i}: k \in F\}$ for all $i<N$ and all
$F \subseteq G_i$ with $|F| \geq m$. Then, for each $i<N$, there is $H_i
\subseteq G_i$ with cardinality at least $A$ such that $\bigcap \{X_{k,i}:
i<N, k \in H_i\} \text{ is infinite.}$
\end{lema}
\begin{proof}
Let $u$ be a nonprincipal ultrafilter in $\N$ which contains $X$. By the
hypothesis, for each $i<N$, there are at most $m-1$ elements $k \in G_i$ for
which $X_{k,i} \notin u$, since if they don't belong to $u$, their union does
not belong to $u$ as well. Taking $H_i=\{k \in G_i: X_{k,i} \in u\}$, we have
that $|H_i| \geq |G_i| - m \geq A$, which concludes the proof.
\end{proof}

\begin{lema}\label{combinatorio4}
Let $N$ be a natural number and for each $i<N$, let $G_i \subseteq \N$ be
finite. For each $k \in \bigcup_{l<N}G_l$ and each $i,j<N$, let
$(m^k_{i,j}(n): n \in G_i)$ be a sequence of positive real numbers such that
\[ \forall i,j<N, \quad \forall k \in \bigcup_{l<N}G_l, \quad \sum_{n \in G_i}
m_{i,j}^k(n) \leq M,\]
for a fixed positive real number $M$. Let $\eta$ be a positive real number and
let $m$ be a natural number such that $m \cdot \eta>M$. Given a natural number
$C$, there is a natural number $B$ depending on $C$, $N$ and $m$ such that if
$|G_i| \geq B$ for all $i<N$, then, for each $i<N$, there is $G_i^* \subseteq
G_i$ with cardinality $C$ such that for all $i,j<N$ with $i \neq j$, all $k
\in G_i^*$ and all $n \in G_j^*$, we have that $m^k_{i,j}(n) < \eta$.
\end{lema}
\begin{proof}
We prove it by induction on $N$. If $N=1$, we are done.

So, assuming that the result holds for $N$, we prove it for $N+1$. Fix a
natural number $C$ and let $C'=C + C\cdot m$. By the inductive hypothesis,
there is $B(N, C')\in \N$ satisfying the lemma for $N$ and $C'$. We claim that
$B(N+1,C) = \max \{C +N^2\cdot C'\cdot m, B(N,C')\}$ works.

Since $B(N+1, C) \geq B(N,C')$, we have that for each $i<N$, there is
$G_i^{**}\subseteq G_i$ with $|G_i^{**}| =C'$ satisfying the thesis. Taking $K
= \bigcup\{G_i^{**}: i<N\}$, we have that $|K|\leq N \cdot C'$. For each $j<N$
and each $k \in K$, there at most $m$ elements $n$ of $G_N$ such that
$m_{N,j}^k(n) \geq \eta$. So, taking
$$G_N^*=\{n \in G_N:m_{N,j}^k(n)<\eta, \text{ for all }k \in K \text{ and all
}j<N \},$$
we have that $|G_N^*| \geq C$ and $m^k_{N,j}(n)<\eta$ for all $k \in K$, all
$j<N$ and all $n \in G_N^*$. We can suppose without loss of generality that
$|G_N^*|=C$. Since for each $i<N$ and each $k\in G_N^*$, there are at most $m$
elements $n$ of $G_i^{**}$ such that $m_{i,N}^k(n) \geq \eta$, taking
$$G_i^*=\{n \in G_i^{**}:m_{i,N}^k(n)<\eta, \text{ for all }k\in
G_N^* \},$$
we have that $|G_i^*| \geq C$ and for each $i,j<N$ with $i \neq j$, each $n
\in G_i^*$ and each $k \in G_j^*$ we have that $m^k_{i,j}(n) < \eta$, which
concludes the proof.
\end{proof}

\begin{prop}\label{PassoIndutivo}
Let $s \in \mathbb{S}^\kappa$, let $E \subseteq \N$ be finite and let
$\varepsilon$ and $M$ be positive real numbers. Let $(\dot{\mu}_k :k \in \N)$
be a sequence of names for finitely additive measures on the Boolean algebra
$\wp(\N)$, let $(\dot{A}_k : k \in \N)$ be a sequence of names for subsets of
$\N$ and let $\dot{X}$ be a name for a subset of $\N$. Suppose that
$$s \Vdash \left\{\begin{array}{l}
\forall k \in \N, \quad \Vert \dot{\mu}_k \Vert < \check{M},\\
\forall k, k'\in \N, \quad k \neq k',\quad \dot{A}_k \cap \dot{A}_{k'}
= \emptyset,\\
\dot{X} \text{ is infinite,}\\
\check{E} \cap \dot{X} = \emptyset,\\
\forall k \in \dot{X}, \quad |\dot{\mu}_k(\dot{A}_k)| \geq \check{\varepsilon}.
\end{array}\right.$$
Given a natural number $N$, a finite subset $F$ of $\kappa$ and a positive
real number $\delta$, there are: $s^* \in \mathbb{S}^\kappa$ with $s^*
\leq_{F,N} s$; $a^* \subseteq \N$; $E^* \subseteq \N$ with $|E^*|\leq
2^{N|F|}$; a sequence of names $(\dot{A}_k^*: k \in \N)$ for subsets of $\N$;
and a name $\dot{X}^*$ for a subset of $\N$ such that $s^*$ forces that:
\begin{enumerate}
\item for all $k \in \N$, $\dot{A}_k^* = \dot{A}_k \setminus \check{a}^*$ and
  so, for all $k, k'\in \N$, $k \neq k'$, $\dot{A}_k^* \cap \dot{A}_{k'}^* =
  \emptyset$;
\item $\dot{X}^* \subseteq \dot{X}$ and $\dot{X}^*$ is infinite;
\item for all $k \in \dot{X}^*$, $|\dot{\mu}_k|(\check{a}^*) \leq
 \check{\delta}$ and so, for all $k \in \dot{X}^*$,
 $|\dot{\mu}_k(\dot{A}_k^*)| \geq \check{\varepsilon} - \check{\delta}$;
\item for all $k \in \check{E}$, $|\dot{\mu}_k|(\check{a}^*) \leq
\check{\delta}$;
\item there is $k \in \check{E}^*$ such that $|\dot{\mu}_k(\check{a}^*)|\geq
 \check{\varepsilon}-\check{\delta}$;
\item $\check{E}^* \subseteq \dot{X} \setminus \dot{X}^*$.
\end{enumerate}
\end{prop}

\begin{proof}
First, take $K=2^{N|F|}$ and $\eta=\frac{\delta}{K}$. Fix a natural number $m$
such that $m \cdot \eta > M$. Let $B \in \N$ be large enough (we need it large
enough in order to have a number greater than $1$ after several applications
of Lemmas \ref{combinatorio} and \ref{combinatorio4}).

We will define $a^*$ as the ``union of some $\dot{A}_k$'s'' with $k \in
\dot{X}$. To find the $k$'s that will work, we have to decide many (but
finitely many) elements of $\dot{X}^*$ and after that, we will eliminate those
which do not serve. So, take $L=l(F, N, s)$ and let
$$D=\{p \in \mathbb{S}^\kappa: \text{ there is } G \subseteq \N
\text{ with }|G|= B \text{ and such that } p \Vdash \check{G} \subseteq
\dot{X}\}.$$
Since $D$ is dense below $s$ and open, by Lemma 1.8 of \cite{BaumgartnerSacks}
there is $s'\leq_{F,N}s$ such that $s'|\sigma \in D$ for all $\sigma \in
L$. Hence for each $\sigma \in L$, there is $G_\sigma \subseteq \N$ with
cardinality $B$ such that $s'|\sigma$ forces that $\check{G}_\sigma \subseteq
\dot{X}$.

Now, we want to decide $\dot{A}_k$ for each $k \in G_\sigma$ and each $\sigma
\in L$. Let $G=\bigcup_{\sigma \in L} G_\sigma$ and
$$D'=\{p \in \mathbb{S}^\kappa: \text{ for all }k \in G \text{ there is
}A_k\subseteq \N \text{ such that } p \Vdash \check{A}_k =
\dot{A}_k\}.$$
Again, since $D'$ is dense below $s'$ and open, applying Lemma 1.8 of
\cite{BaumgartnerSacks} we obtain $s''\leq_{F,N}s'$ such that for each $\sigma
\in L$ we have $s''|\sigma \in D'$. Hence for each $\sigma \in L$ and each $k
\in G$, there is $A_k(\sigma) \subseteq \N$ such that $s''|\sigma$ forces that
$\check{A}_k(\sigma)=\dot{A}_k$ and therefore, for each $\sigma \in L$,
$(A_k(\sigma))_{k \in G}$ is pairwise disjoint.

Since we want the measures of $a^*$ to satisfy properties $(3)$, $(4)$ and
$(5)$, and the names $\dot{A}_n^*$ to satisfy properties $(1)$ and $(3)$, we
will approximate the values of the measures $\dot{\mu}_k$ for $k \in E\cup G$
in the sets $A_n(\sigma)$ for $\sigma \in L$ and $n \in G_\sigma$.
\vskip 5pt
\noindent {\bf Claim 1.} \emph{For each $\sigma, \sigma' \in L$, each $k \in
  E\cup G$ and each $n \in G_{\sigma}$, there is $m^k_{\sigma,\sigma'}(n) \in
  \mathbb{R}$ and there is $t \leq_{F,N} s''$ such that for all $\sigma,
  \sigma'\in L$, all $k \in E\cup G$ and all $n \in G_\sigma$, $t|\sigma$
  forces that $|\dot{\mu}_k|(\check{A}_n(\sigma')) \leq
  \check{m}^k_{\sigma,\sigma'}(n)$ and such that for all $\sigma, \sigma'\in
  L$ and all $k \in E\cup G$,
$$\sum_{n \in G_{\sigma}}m^k_{\sigma,\sigma'}(n) < M.$$}
\vskip 5pt
\noindent \emph{Proof of Claim 1.}
Since $E\cup G$ is finite and $s$ forces that
$\Vert\dot{\mu}_k\Vert<\check{M}$ for all $k \in \check{E}\cup \check{G}$, it
forces also that there is $\theta >0$ such that $\check{M}-\Vert \dot{\mu}_k
\Vert \geq \theta$ for all $k \in \check{E}\cup \check{G}$. Thus, there is $p
\leq_{F,N}s''$ such that $p$ forces that $\check{M} - \Vert \dot{\mu}_k \Vert
\geq \check{\theta}$ for all $k \in E\cup G$.

There is $t \leq_{F,N} p$ and for each $\sigma, \sigma' \in L$, each $n \in
G_{\sigma}$ and each $k \in E\cup G$, there is $m^k_{\sigma,\sigma'}(n) \in
\mathbb{R}$ such that
$$t|\sigma \Vdash 0 \leq
\check{m}^k_{\sigma,\sigma'}(n)-|\dot{\mu}_k|(\check{A}_n(\sigma')) <
\frac{\check{\theta}}{2|G_\sigma|}.$$
Recall that for each $\sigma \in L$, $(A_k(\sigma))_{k \in G}$ is pairwise
disjoint. Therefore, for all $\sigma,\sigma'\in L$ and all $k \in \check{E}\cup
 \check{G}$, $t|\sigma$ forces that
$$\check{M}- \sum_{n \in G_\sigma}
\check{m}^k_{\sigma,\sigma'}(n) \geq \check{M}-
\sum_{n \in G_\sigma}(|\check{\mu}_k|(\check{A}_n(\sigma')) -
\frac{\check{\theta}}{2|G_\sigma|}) \geq \check{M} - \Vert \check{\mu}_k
\Vert - \frac{\check{\theta}}{2} \geq \frac{\check{\theta}}{2}>0$$
and so, $\sum_{n \in G_\sigma} m^k_{\sigma,\sigma'}(n) < M$. Moreover, for
each $n \in G_{\sigma}$, each $\sigma,\sigma'\in L$ and each $k \in
\check{E}\cup \check{G}$, $t|\sigma$ forces that
$|\dot{\mu}_k|(\check{A}_n(\sigma')) \leq \check{m}^k_{\sigma,\sigma'}(n)$,
concluding the proof of the claim. $\square$

From the fact that for each $\sigma, \sigma'\in L$ and each $k \in E$,
$\sum_{n \in G_{\sigma}}m^k_{\sigma,\sigma'}(n) < M$, it follows that for each
$\sigma, \sigma'\in L$ and each $k \in E$ there are at most $m$ elements $n$
of $G_\sigma$ such that $m^k_{\sigma,\sigma'}(n) \geq \eta$. Since $B$ is
large enough, we can assume without loss of generality that
\[\forall \sigma,\sigma' \in L, \quad \forall k \in E, \quad \forall n 
\in G_\sigma, \quad m^k_{\sigma,\sigma'}(n) <\eta,\eqno{(7)}\]
and by Lemma \ref{combinatorio4}, we can assume without
loss of generality that
\[\forall \sigma, \sigma'\in L, \quad \sigma \neq \sigma', \quad \forall n \in
G_{\sigma}, \quad \forall k \in G_{\sigma'}, \quad m^k_{\sigma,\sigma'}(n) <
\eta.\eqno{(8)}\]

To obtain $\dot{X}^*$ satisfying $(2)$, for each $\sigma
\in L$ and each $n \in G_\sigma$, let $\dot{X}_{n,\sigma}$ be a name 
such that $t$ forces that $\dot{X}_{n,\sigma}=\{k \in \dot{X}: |\dot{\mu}_k|
(\check{A}_n(\sigma)) < \eta\}$.
\vskip 5pt
\noindent {\bf Claim 2.} \emph{ There is $s^* \leq_{F,N} t$ and for
  each $\sigma \in L$ there is a nonempty $H_\sigma \subseteq G_\sigma$ such
  that $s^*$ forces that $\bigcap\{\dot{X}_{n,\sigma} : n \in
\check{H}_{\sigma}, \sigma \in L\} \text{ is infinite.}$}
\vskip 5pt
\noindent \emph{Proof of Claim 2.} Let $(\sigma_i :i<K)$ be an enumeration of
$L$. To prove the claim, we will proceed by
induction on $j<K$ to construct a sequence of conditions $s^j$ such that
$s^{j+1} \leq_{F,N} s^j \leq_{F,N} t$ and for each $\sigma \in L$, we
construct a sequence of finite nonempty sets $H_{\sigma,j}$ with
$H_{\sigma,j+1} \subseteq H_{\sigma,j} \subseteq G_{\sigma}$ such that
$s^{j}|\sigma_j$ forces that $\bigcap\{\dot{X}_{n,\sigma} : n \in
\check{H}_{\sigma,j}, \sigma \in L \} \text{ is infinite.}$

For the construction, fix $0\leq j < K$ and suppose we already have $s^j$ and
$H_{\sigma,j}$ for all $\sigma \in L$ as wanted. We have that $s^j$ forces that
$\dot{X} \subseteq \bigcup\{\dot{X}_{n,\sigma}:n \in H\}$ for all $\sigma \in
L$ and all $H \subseteq H_{\sigma,j}$ with $|H|\geq m$, for if not, then there
would be $\sigma \in L$, $k \in \N$, $H \subseteq H_{\sigma,j}$ with $|H|
\geq m$ and $(s^j)'\leq s^j$ such that $(s^j)'$ forces that $k \in \dot{X}
\setminus \bigcup\{ \dot{X}_{n,\sigma}: n \in H\}$. Since for each $\sigma \in
L$, $(A_n(\sigma))_{n \in H}$ are pairwise disjoint, we would have that
$(s^j)'$ forces that $|\dot{\mu}_k| (\bigcup\{ \check{A}_n(\sigma):n \in H\})
\geq \sum_{n \in H}|\dot{\mu}_k| (\check{A}_n(\sigma)) \geq \check{m} \cdot
\check{\eta} > \check{M}$, contradicting our hypothesis. We apply Lemma
\ref{combinatorio} in $V[G]$ and using Lemma 1.8 of \cite{BaumgartnerSacks} we
have that there is $s^{j+1} \leq_{F,N} s^j$ and for each $\sigma \in L$ there
is $H_{\sigma,j+1} \subseteq H_{\sigma,j}$, such that $s^{j+1}|\sigma_{j+1}$
forces that $\bigcap \{\dot{X}_{n,\sigma}: \sigma \in L, n \in
\check{H}_{\sigma,j+1} \} \text{ is infinite.}$
Since $C$ is large enough, we can assume each $H_{\sigma,K-1}$ to be nonempty
and taking $s^*=s^{K-1}$ and $H_\sigma = H_{\sigma,K-1}$ we conclude the proof
of the claim. $\square$

For each $\sigma \in L$ we take $k_\sigma \in H_{\sigma}$ of Claim 2. We
define $a^* = \bigcup_{\sigma \in L} A_{k_\sigma}(\sigma)$ and $E^*
=\{k_\sigma : \sigma\in L\}$. For each $k \in \N$, we define $\dot{A}_k^*$
names for $\dot{A}_k \setminus a^*$ and $\dot{X}^*$ a name such that 
$$s^*\Vdash \dot{X}^* = \bigcap \{\dot{X}_{k,\sigma}:k
\in \check{H}_{\sigma},\sigma \in L\} \setminus \check{E}^*.$$
\vskip 5pt
\noindent {\bf Claim 3.}
$s^*$ forces that for all $k \in \bigcap\{\dot{X}_{n,\sigma'} : n \in
\check{H}_{\sigma'}, \sigma' \in L\}$ and all $\sigma \in L$, 
$|\dot{\mu}_k|(\check{A}_{k_\sigma}(\sigma)) < \check{\eta}$.
\vskip 5pt
\noindent \emph{Proof of Claim 3.} Suppose that $s^*$ forces that $k \in
\dot{X}_{n,\sigma}$ for each $n \in H_\sigma$ and each $\sigma \in L$. Then,
by the definition of $\dot{X}_{n,\sigma}$ it means that $s^*$ forces that
$|\dot{\mu}_k|(\check{A}_n(\sigma))<\check{\eta}$ for each $n \in H_\sigma$
and each $\sigma \in L$. Since each $k_\sigma \in H_{\sigma}$, we conclude the
claim. $\square$

Let us now verify that we have everything we wanted: first, note that by the
definition of $\dot{A}_k^*$ we have that $s^*$ forces that $\dot{A}_k^*=
\dot{A}_k \setminus \check{a}^*$ and $(\dot{A}_k^*)_{k \in
  \N}$ are pairwise disjoint, since it forces that
$(\dot{A}_k)_{k \in \N}$ are pairwise disjoint. Therefore we obtain
$(1)$.

By the definition of $\dot{X}^*$ we have that $s^*$ forces that $\dot{X}^*
\subseteq \dot{X}$. By Claim 2 we have that $s^*$ forces that $\dot{X}^*$ is
infinite, since $E^*$ is finite. So we obtain $(2)$. 

By Claim 3 we have that
$$s^* \Vdash \forall k \in \dot{X}^*, \quad
|\dot{\mu}_k|(\check{a}^*) = |\dot{\mu}_k|(\bigcup_{\sigma \in L}
\check{A}_{k_\sigma}(\sigma)) \leq \sum_{\sigma \in L} |\dot{\mu}_k|
(\check{A}_{k_\sigma}(\sigma)) \leq \check{K} \cdot \check{\eta} \leq 
\check{\delta}.$$
By the hypothesis of the proposition, $s$ forces that
$|\dot{\mu}_k(\dot{A}_k)| \geq \check{\varepsilon}$ for each $k \in
\dot{X}$. To obtain $(3)$, note that
$$s^* \Vdash \forall k \in \dot{X}^*, \quad |\dot{\mu}_k(\dot{A}_k^*)| \geq
|\dot{\mu}_k(\dot{A}_k)| - |\dot{\mu}_k|(\check{a}^*) \geq \check{\varepsilon}
- \check{\delta}.$$

To verify $(4)$, note that Claim 1 and $(7)$ imply that for all $\sigma,
\sigma'\in L$, all $k \in E$ and all $n \in G_\sigma$
$$s^*|\sigma \Vdash |\dot{\mu}_k|(\check{A}_n(\sigma')) \leq
\check{m}^k_{\sigma,\sigma'}(n) < \check{\eta}.$$
Since for all $\sigma \in L$, $k_\sigma \in H_\sigma \subseteq G_\sigma$, then,
for each $\sigma \in L$ and each $k \in E$ we have that
$$s^*|\sigma \Vdash |\dot{\mu}_k|(\check{a}^*) =
|\dot{\mu}_k|(\bigcup_{\sigma' \in L} \check{A}_{k_{\sigma'}}(\sigma')) \leq
\sum_{\sigma' \in L} |\dot{\mu}_k|(\check{A}_{k_{\sigma'}}(\sigma')) \leq
\check{K} \cdot \check{\eta} \leq \check{\delta},$$
and so, by Lemma 1.9 of \cite{BaumgartnerSacks} we obtain $(4)$.

To verify $(5)$, note that Claim 1 and $(8)$ imply that for all $\sigma,
\sigma'\in L$ with $\sigma \neq \sigma'$, all $n \in G_\sigma$ and all $k \in
G_\sigma'$
$$s^*|\sigma \Vdash |\dot{\mu}_k|(\check{A}_n(\sigma')) \leq
\check{m}^k_{\sigma,\sigma'}(n) < \check{\eta}.$$
Since for all $\sigma \in L$, $k_\sigma \in H_\sigma \subseteq G_\sigma$, we
have that for each $\sigma \in L$,
$$s^*|\sigma \Vdash |\dot{\mu}_{k_\sigma}(a^*)| \geq
|\dot{\mu}_{k_\sigma}(\check{A}_{k_\sigma}(\sigma))| - \sum_{\sigma'\in L,
  \sigma'\neq \sigma}|\dot{\mu}_{k_\sigma}|(\check{A}_{k_{\sigma'}}(\sigma')) 
   > \check{\varepsilon} - \check{K}\cdot\check{\eta} = 
    \check{\varepsilon} - \check{\delta},
$$
and again by Lemma 1.9 of \cite{BaumgartnerSacks} we obtain $(5)$.

By the definition of $E^*$ and that of $\dot{X}^*$ we have $(6)$. 
\end{proof}

\section{The proof of the main theorem}

We show now how the main result follows from Proposition \ref{PassoIndutivo}.

\begin{teo}
Let $\kappa>\omega_1$ be a regular cardinal. Let $G$ be an
$\mathbb{S}^\kappa$-generic filter over a set-theoretic universe $V$ where CH
holds. In $V[G]$, if $K$ is the Stone space of the Boolean algebra $\wp(N)
\cap V$, then $C(K)$ is a Banach space with the Grothendieck property and
density $\omega_1$ which is less than $\kappa=\cc$.
\end{teo}

\begin{proof}
First we work in $V[G]$. By Theorem 1.11 of \cite{BaumgartnerSacks},
$|\wp(\N) \cap V|= \omega_1$ and by Theorem 1.14 of \cite{BaumgartnerSacks},
$\omega_1<\kappa=\cc$. So, $K$ has weight $\omega_1$ and $C(K)$ has density
$\omega_1$, which is less than $\cc$.

Now suppose that $C(K)$ is not a Grothendieck space. Then there is a sequence
$(\mu_k)_{k \in \N} \subseteq C(K)^*$ which is weak$^*$ convergent to
$\mu \in C(K)^*$ but does not converge weakly. If $\{\mu_k: k \in \N\}$ were 
weakly compact, by the
  Eberlein-\v Smulian Theorem, it would be sequentially weakly
  compact. Then there would be infinite and disjoint sets $M_1, M_2 \subseteq
  \N$ such that $(\mu_k)_{k \in M_i}$ is weakly convergent to $\nu_i$ for
  $i=1,2$, and $\nu_1 \neq \nu_2$. Since weak convergence implies weak$^*$
  convergence, we would have that $(\mu_k)_{k \in M_i}$ converges weakly$^*$
  to $\nu_i$ for $i=1,2$, a contradiction. 

So, we can assume that $\{\mu_k: k \in \N\}$ is not weakly compact. By the
Uniform Boundedness Principle, $(\mu_k)_{k \in \N}$ is a bounded sequence. By
the Dieudonn\'e-Grothendieck Theorem (Theorem VII.14 in \cite{Diestel}), there
is a pairwise disjoint sequence $(U_k)_{k \in \N}$ of open subsets of $K$ and
$\varepsilon >0$ such that for all $k_0 \in \N$ there is $k \geq k_0$ and $n_k
\in \N$ such that $|\mu_{n_k}(U_k)| \geq \varepsilon$. Since $K$ is a Boolean
space, using the regularity of each $\mu_k$, we can assume without loss of
generality that $U_k =B_k$ for some $B_k \in \mathcal{B}$. Moreover, if for
some $k \in \N$ we have that $|\mu_k(B_{k_i})|\geq \varepsilon$ for some
sequence $(k_i)_{i \in \N} \subseteq \N$, it follows that $|\mu_k|(\bigcup_{i
  \in \N}B_{k_i}) \geq \sum_{i \in \N}|\mu_k(B_{k_i})| = \infty$,
contradicting the fact that $\mu_k$ is bounded. So, let $i_0 \in \N$ and $n_0
\in \N$ be such that $|\mu_{n_0}(B_{i_0})|\geq \varepsilon$ and we construct
by induction $i_{k+1} > i_k$ and $n_{k+1} > n_k$ such that
$|\mu_{n_k}(B_{i_k})| \geq \varepsilon$. Let $A_k=B_{i_k}$ and we have that
$(A_k)_{k \in \N}\subseteq \wp(\N)\cap V$ are pairwise disjoint and $(n_k)_{k
  \in \N} \subseteq \N$ is an increasing sequence such that for all $k \in
\N$, $|\mu_{n_k}(A_k)|\geq \varepsilon$.

Working now in $V$, let $\dot{\mu}_k$ be a name for the restriction of
$\mu_{n_k}$ to the Boolean algebra $\wp(\N)\cap V$. Let $s \in
\mathbb{S}^\kappa$, let $M$ and $\varepsilon$ be positive real numbers and let
$\dot{A}_k$ be names for the elements of $\wp(\N)\cap V$ such that
$$s \Vdash \left\{ \begin{array}{l}
\forall k \in \N, \quad \Vert \dot{\mu}_k \Vert \leq \check{M},\\
\forall k, k' \in \N, \quad k \neq k', \quad \dot{A}_k \cap \dot{A}_{k'}= 
\emptyset,\\
\forall k \in \N, \quad |\dot{\mu_k}(\dot{A}_k)| \geq \check{\varepsilon}.
\end{array}\right.$$

By induction, we will construct a sequence $(s_N)_{N
  \in \N}$ with $s_{N+1}\leq_{F_N,N} s_N$ where $F_N =\{\alpha_i^k : i,k
<N\}$ and $supp(s_N)= \{\alpha^N_k: k \in \N\}$, a pairwise disjoint
  sequence $(a_N)_{N \in \N}$ in $\wp(\N)$, a pairwise
  disjoint sequence $(E_N)_{N \in \N}$ of finite subsets of
  $\N$, sequences of names $(\dot{A}^N_k)_{k \in \N, N \in \N}$ for
  subsets of $\N$ and a sequence of names $(\dot{X}_N)_{N \in \N}$ for
  subsets of $\N$ such that $s_{N+1}$ forces that
\begin{enumerate}
\item for all $k \in \N$, $\dot{A}_k^{N+1} = \dot{A}_k^N \setminus
\check{a}_{N+1}$ and so, for all $k, k'\in \N$ if $k \neq k'$,
then $\dot{A}_k^{N+1} \cap \dot{A}_{k'}^{N+1} = \emptyset$;
\item $\dot{X}_{N+1} \subseteq \dot{X}_N$ and $\dot{X}_{N+1}$ is infinite;
\item for all $k \in \dot{X}_{N+1}$, $|\dot{\mu}_k|(\check{a}_{N+1}) \leq
 \check{\delta}_{N}$ and so, for all $k \in \dot{X}_{N+1}$,
$|\dot{\mu}_k(\dot{A}_k^{N+1})| \geq \check{\varepsilon}_{N} -
 \check{\delta}_N$;
\item for all $k \in \bigcup_{0 \leq i \leq N}\check{E}_i$, 
$|\dot{\mu}_k|(\check{a}_{N+1}) \leq \check{\delta}_N$;
\item there is $k \in \check{E}_{N+1}$ such that
  $|\dot{\mu}_k(\check{a}_{N+1})|\geq  \check{\varepsilon}_N-\check{\delta}_N$;
\item $\check{E}_{N+1} \subseteq \dot{X}_N \setminus \dot{X}_{N+1}$;
\end{enumerate}
where $\varepsilon_0=\varepsilon$, $\delta_0=\frac{\varepsilon}{2^3}$ and for
each $N\in \N$, $\varepsilon_{N+1}=\varepsilon_N - \delta_N$ and
$\delta_{N+1}= \frac{\delta_N}{2}$.

For the construction, let $s_0=s$, $E_0= \emptyset$, $\dot{A}^0_k =
\dot{A}_k$, $\dot{X}_0= \check{\N}$ and note that we have the hypothesis of
Proposition \ref{PassoIndutivo}.

Now suppose we already have $s_0, \dots, s_N$, $E_0, \dots, E_N$, 
$(\dot{A}^N_k)_{k \in \N}$, $\dot{X}_0, \dots, \dot{X}_N$, and
$a_1, \dots,$ $a_N$ as we want. Note that $(1), (2), (3)$ and $(6)$
guarantee that the hypothesis of Proposition
\ref{PassoIndutivo} is satisfied. Then there are $s_{N+1} \leq_{F_N, N} s_N$,
$a_{N+1} \subseteq \N$, $E_{N+1} \subseteq \N$, a sequence of names 
$(\dot{A}^{N+1}_k)_{k \in \N}$ for subsets of $\N$ and a name 
$\dot{X}_{N+1}$ for a subset of $\N$ satisfying $(1)-(6)$. This concludes the
construction of the sequences.

Then, by Lemma 1.6 of \cite{BaumgartnerSacks} there is $s^*\in
\mathbb{S}^\kappa$ such that $s^*\leq s_N$ for all $N \in \N$. We have
that $(1)$ guarantees that $(a_N)_{N \in \N}$ are pairwise disjoint. Moreover,
$(4)$ guarantees that
$$s^* \Vdash \forall N \in \N, \ \text{if } i<N \text{ and } k \in
\check{E}_i, \text{ then } |\dot{\mu}_k|(\check{a}_{N}) \leq
\frac{\check{\varepsilon}}{2^{N+3}}. \eqno{(9)}$$
On the other hand, using $(3)$ and $(6)$ we conclude that
$$s^* \Vdash \forall N \in \N, \ \text{if } i>N \text{ and } k \in
\check{E}_i, \text{ then } |\dot{\mu}_k|(\check{a}_{N}) \leq
\frac{\check{\varepsilon}}{2^{N+3}}. \eqno{(10)}$$
And finally we have that $(5)$ guarantees that
$$s^* \Vdash \forall N \in \N, \ \exists k \in \check{E}_N, \
|\dot{\mu}_k(\check{a}_{N})| \geq \frac{3\check{\varepsilon}}{4}.\eqno{(11)}$$

Let $(K_\alpha)_{\alpha<\omega_1}
\subseteq \wp(\N)$ be an almost disjoint family. For each $\alpha <
\omega_1$ we have that
$$s^* \Vdash \bigvee_{N \in \check{K}_\alpha}\check{a}_N \in
\wp(\N)\cap V.$$
By Theorem 1.11 of \cite{BaumgartnerSacks}, $s^*$ forces that
  $(\check{K}_\alpha)_{\alpha<\check{\omega}_1}$ is an almost disjoint family
  of subsets of $\N$ and $\check{\omega}_1=\omega_1$. By Lemma
  \ref{combRosenthal} applied in $V[G]$, we obtain $s^{**} \leq  s^*$ and
  $\alpha \in \omega_1$ such that
$$s^{**} \Vdash \forall k \in \N \quad \dot{\mu}_k(\bigvee_{N \in
  \check{K}_\alpha}\check{a}_N) = \sum_{N \in \check{K}_\alpha}
  \dot{\mu}_k(\check{a}_N).$$

Take $a=\bigvee_{N \in K_\alpha}a_N$
and let us see that in $V[G]$, if $\delta =\frac{\varepsilon}{4}$, then there
are infinitely many $n \in \N$ such that $|\mu_n(a)| \geq 2 \delta$ and
infinitely many $l \in \N$ such that $|\mu_l(a)| \leq \delta$.If $i \in
K_\alpha$, using $(9)$, $(10)$ and $(11)$, we have that $s^{**}$ forces that
there is $k \in \check{E}_i$ such that
$$ |\dot{\mu}_k(\check{a})| =|\sum_{N \in \check{K}_\alpha}
\dot{\mu}_k(\check{a}_N)| \geq |\dot{\mu}_k(\check{a}_i)| - \sum_{N \in
  \check{K}_\alpha \setminus \{i\}} |\dot{\mu}_k|(\check{a}_N)
 \geq \frac{3\check{\varepsilon}}{4} - \frac{\check{\varepsilon}}{4} =
 2\check{\delta}.$$
On the other hand, if $i \notin K_\alpha$, then
$$s^{**} \Vdash \forall k \in \check{E}_i \quad 
|\dot{\mu}_k(\check{a})| =|\sum_{N \in \check{K}_\alpha}
\dot{\mu}_k(\check{a}_N)| \leq \sum_{N \in \check{K}_\alpha}
\frac{\check{\varepsilon}}{2^{N+3}} \leq \frac{\check{\varepsilon}}{4} =
\check{\delta}.$$

Since $a \in \wp(\N)\cap V$ which is identified with $Clop(K)$, we have that
$\chi_a \in C(K)$ and therefore $(\mu_k)_{k \in \N}$ does not
converge weakly$^*$, contradicting our hypothesis and concluding the proof.
\end{proof}

\bibliographystyle{abbrv}

\end{document}